\documentclass{article}

\usepackage{amsmath,amssymb}
\usepackage{amscd}
\usepackage{mathrsfs}
\usepackage{amsthm}

\theoremstyle{plain}

\newtheorem{theorem}{\bf Theorem}[section]
\newtheorem{lemma}[theorem]{\bf Lemma}
\newtheorem{proposition}[theorem]{\bf Proposition}

\theoremstyle{definition}
\newtheorem{definition}[theorem]{Definition}
\newtheorem{example}[theorem]{\bf Example}
\newtheorem{remark}[theorem]{\bf Remark}
\newtheorem{problem}[theorem]{\bf Problem}
\newtheorem{question}[theorem]{\bf Question}

\newcommand{\eqa}[1]{
\begin{align*}
#1
\end{align*}}

\usepackage[all]{xy}

\newcommand{\nai}[2]{\langle #1,#2\rangle}
\newcommand{\dom}[1]{{{\rm{dom}}{(#1)}}}

\newcommand{\wlimn}[1]{{\rm{w}}\text{-}\lim_{n\to \infty}}

\title{On Polish Groups of Finite Type}

\author{Hiroshi Ando$^{1,2}$\\
$^{1\ }$University of Copenhagen\\
Universitetsparken 5, 2100 K\o benhavn \O, Denmark\\
$^{2\ }$Research Institute for Mathematical Sciences, Kyoto University\\
Kyoto, 606-8502, Japan\\
E-mail: andonuts@kurims.kyoto-u.ac.jp\\
\\
Yasumichi Matsuzawa$^{3,4}$\\
$^{3\ }$Mathematisches Institut, Universit\"{a}t Leipzig\\
Johannisgasse 26, 04103, Leipzig, Germany\\
$^{4\ }$Department of Mathematics, Hokkaido University\\
Kita 10, Nishi 8, Kita-ku, Sapporo, 060-0810, Japan\\
E-mail: matsuzawa@math.sci.hokudai.ac.jp
}

\begin{document}
\maketitle
\begin{abstract}
Sorin Popa initiated the study of Polish groups which are embeddable into the unitary group of a separable finite von Neumann algebra. Such groups are called of finite type or said to belong to the class $\mathscr{U}_{\text{fin}}$. We give necessary and sufficient conditions for Polish groups to be of finite type, and construct exmaples of such groups from I$_{\infty}$ and II$_{\infty}$ von Neumann algebras. We also discuss permanence properties of finite type groups under various algebraic operations. Finally we close the paper with some questions concerning Polish groups of finite type. 
\end{abstract}


\medskip


\medskip

\pagebreak 
\tableofcontents

\section{Introduction}
\label{intro}
In this paper we consider the following problem. Denote by $\mathcal{U}(M)$ the unitary group of a von Neumann algebra $M$. 
\begin{problem}
Determine the necessary and sufficient condition for a Polish group $G$ to be isomorphic as a topological group onto a strongly closed subgroup of some $\mathcal{U}(M)$, where $M$ is a separable finite von Neumann algebra.
\end{problem}
S. Popa defined a Polish group to be of finite type if it has this embedding property. Denote by $\mathscr{U}_{\text{fin}}$ the class of all finite type Polish groups. He initiated the study of this class in an attempt to enrich the study of rapidly developing cocycle superrigidity theory (cf. \cite{Furman, Peterson, Popa}). In particular, he proposed in \cite{Popa} the problem of studying and characterizing  the class $\mathscr{U}_{\text{fin}}$. 

Secondly, this problem is motivated from our previous work \cite{AndoMats} on infinite-dimensional Lie algebras associated with such groups: Let $M$ be a finite von Neumann algebra on a Hilbert space $\mathcal{H}$. Let $G$ be a strongly closed subgroup of $\mathcal{U}(M)$ and $\overline{M}$ be a set of all densely defined closed operators on $\mathcal{H}$ which are affiliated to $M$. It is proved that the set 
\[\text{Lie}(G):=\{A^*=-A\in \overline{M}; e^{tA}\in G\text{ for all }t\in \mathbb{R}\}\]
is a complete topological Lie algebra with respect to the strong resolvent topology (see also the related work of D. Beltita \cite{Daniel}).  
Since these Lie algebras turn out to be non-locally convex in general when $M$ is non-atomic, they are quite exotic as a Lie algebra and their properties are still unknown. Therefore it would be interesting to find non-trivial examples of such groups. 


We give an answer in Theorem \ref{yasu characterization of finite type} to the above Problem by the aid of positive definite functions on groups and their GNS representations, and characterize locally compact groups or amenable Polish groups of finite type via compatible bi-invariant metrics in Proposition \ref{yasu locally compact} and Theorem \ref{yasu amenable group} (the former is known, but we give a new proof). 
Combining with Popa's result \cite{Popa}, Theorem \ref{yasu characterization of finite type} gives a necessary and sufficient condition for a Polish group to be isomorphic onto a closed subgroup of the unitary group of a separable II$_1$ factor. 
We then give examples of Polish groups $G$ of finite type using noncommutative integration of E. Nelson \cite{Nelson}. 
Finally we discuss some hereditary properties of finite type groups and pose some questions concerning Polish groups of finite type.\\ \\
\textbf{Notation.} In this paper we often say a von Neuman algebra $M$ is {\it separable} if it has a separable predual, especially when the Hilbert space on which $M$ acts is implicit. This is known to be equivalent to the condition that $M$ has a faithful representation on a separable Hilbert space. We denote by $\text{Proj}(M)$ the lattice of all projections in $M$. 
A von Neumann algebra is said to be {\it finite} if it admits no non-unitary isometry.
When we consider a group $G$, its identity is denoted as $e_G$. However, we also use $1$ as the identity when we consider a concrete subgroup of the unitary group of a von Neumann algebra. We always regard the unitary group of a von Neumann algebra as a topological group with the strong operator topology.      
\section{Polish Groups of Finite Type and its Characterization}
In this section, we characterize Polish groups of finite type via
positive definite functions. We then characterize when locally compact
groups or amenable Polish groups are of finite type via compatible bi-invariant metrics. To this end, we
review notions of SIN-groups, bi-invariant metrics and unitary
representability. 
\subsection{Polish Groups of Finite Type}
Recall that a Polish space is a separable completely metrizable
topological space, and a Polish group is a topological group whose
topology is Polish.

We now introduce finite type groups after Popa \cite{Popa}.
\begin{definition}\label{yasu definition of finite type}
A Hausdorff topological group is called of {\it finite type} if it
is isomorphic as a topological group onto a closed subgroup of the
unitary group of a finite von Neumann algebra.
\end{definition}
\begin{remark}
Popa \cite{Popa} requires the topological group of finite
type to be Polish, whereas our definition of finiteness does not
require any countability. 
We will show in Theorem \ref{yasu characterization of finite type} that a Polish group $G$ of finite type in our sense coincides with Popa's definition of finite type group. 
That is, $G$ is isomorphic onto a closed subgroup of the unitary group of a finite von Neumann algebra acting on a separable Hilbert space.
\end{remark}
All of second countable locally  compact Hausdorff groups, 
the unitary group of a von Neumann algebra acting on a separable Hilbert space 
are Polish groups.
Furthermore, separable Banach spaces are Polish groups as an additive group.  
We denote the class of all Polish groups of finite type by $\mathscr{U}_{{\rm fin}}$.





Note that since a von Neumann algebra is finite if and only if its unitary group is complete with respect to the left uniform structure, 
Polish groups of finite type are necessarily complete.
Thus we have the following simple consequence.

\begin{proposition}\label{yasu finite iff fintie}
The unitary group of a von Neumann algebra $M$ acting on a separable Hilbert space
is of finite type if and only if $M$ is finite.
\end{proposition}

Another examples of Polish groups of finite type are given later.

\subsection{Positive Definite Functions}

A complex valued function $f$ on a Hausdorff topological group $G$ is called {\it positive definite} 
if for all $g_1,\cdot\cdot\cdot,g_n\in G$ and for all $c_1,\cdot\cdot\cdot,c_n\in\mathbb{C}$,
\begin{equation*}
\sum_{i,j=1}^{n}\bar{c_i}c_jf(g_i^{-1}g_j)\geq 0
\end{equation*}
holds. 
Moreover if a complex valued function $f$ is invariant under inner automorphisms, 
that is
\begin{equation*}
f(hgh^{-1}) = f(g), \ \ \ \ \forall g,h\in G,
\end{equation*}
then $f$ is called {\it a class function}.

It is well-known that there is an one-to-one correspondence between 
the set of all continuous positive definite functions on a topological group 
and the set of unitary equivalence classes of all cyclic unitary representations of it.
more precisely, for each continuous positive definite function $f$ on a topological group $G$,
there exists a triple $(\pi_f,\mathcal{H}_f,\xi_f)$ consisting of a cyclic unitary representation $\pi_f$ in a Hilbert space $\mathcal{H}_f$
and a cyclic vector $\xi_f$ in $\mathcal{H}_f$ such that 
\begin{equation*}
f(g) = \nai{\xi_f}{\pi_f(g)\xi_f}, \ \ \ \ g\in G,
\end{equation*}
and this triple is unique up to unitary equivalence.
This triple is called the {\it GNS triple} associated to $f$.
Note that if $G$ is separable, then so is $\mathcal{H}_f$.

The GNS triple is of the following form for each continuous positive definite class function.

\begin{lemma}\label{yasu GNS for invariant positive definite functions}
Let $f$ be a continuous positive definite class function on a topological group $G$ 
and $(\pi,\mathcal{H},\xi)$ be its GNS triple.
Then the von Neumann algebra $M$ generated by $\pi(G)$ is finite and the linear functional
\begin{equation*}
\tau(x) := \nai{\xi}{x\xi}, \ \ \ \ x\in M,
\end{equation*}
is a faithful normal tracial state on $M$.
In particular $M$ is countably decomposable.
\end{lemma}

\begin{proof}
It is clear that $\tau$ is a normal state on $M$. Since $f$ is a class function, it is easy to see that $\tau$ is tracial on the strongly dense *-subalgebra of $M$ spanned by $\pi (G)$. Therefore by normality, $\tau$
is tracial on $M$. Therefore we have only to check the faithfulness of $\tau$. 
Assume $\tau(x^*x)=0$.
Since $\tau$ is a trace, we have
\begin{equation*}
\|x\pi(g)\xi\|^2 = \tau(\pi(g)^*x^*x\pi(g)) = 0,
\end{equation*}
for all $g\in G$.
By  the cyclicity of $\xi$, $x$ must be $0$.
\end{proof} 

\begin{example}[I. J. Schoenberg \cite{Schoenberg}]
\label{yasu hilbert space positive definite function}
Let $\mathcal{H}$ be a complex Hilbert space.
Note that $\mathcal{H}$ is an additive group.
Then a function $f$ defined by $f(\xi):=e^{-\|\xi\|^2}\ (\xi\in\mathcal{H})$
is a positive definite (class) function on $\mathcal{H}$. 
\end{example}

\begin{example}[I. J. Schoenberg \cite{Schoenberg}]
\label{yasu l^p positive definite}
For all $1\leq p \leq 2$ a function $f_p$ defined by 
$f_p(a):=e^{-\|a\|_p^p}\ (a\in l^p)$ is a positive definite (class) function 
on a separable Banach space $l^p$.
\end{example}

For more details about positive definite class functions, see \cite{Takeshi}.

\subsection{The First Characterization}

We now characterize Polish groups of finite type.

\begin{theorem}\label{yasu characterization of finite type}
For a Polish group $G$ the following are equivalent.

\begin{list}{}{}
\item[(i)] $G$ is of finite type.
\item[(ii)] $G$ is isomorphic as a topological group onto a closed subgroup of the unitary group of a finite von Neumann algebra acting on a separable Hilbert space. 
\item[(iii)] A family $\mathcal{F}$ of continuous positive definite class functions on $G$ 
generates a neighborhood basis of the identity $e_G$ of $G$.
That is, for each neighborhood $V$ of the identity, there are functions $f_1,\cdot\cdot\cdot,f_n\in \mathcal{F}$
and open sets $\mathcal{O}_1,\cdot\cdot\cdot,\mathcal{O}_n$ in $\mathbb{C}$ such that
\begin{equation*}
e_G\in\bigcap_{i=1}^{n}f_i^{-1}(\mathcal{O}_i)\subset V.
\end{equation*} 
\item[(iv)] There exists a positive, continuous positive definite class function which generates a neighborhood basis of the identity of $G$.
\item[(v)] A family $\mathcal{F}$ of continuous positive definite class functions on $G$ separates the identity of $G$ and closed subsets $A$ with $A\not\ni e_G$.
That is, for each closed subset $A$ with $A\not\ni e_G$, 
there exists a continuous positive definite class function $f\in\mathcal{F}$ such that
\begin{equation*}
\sup_{x\in A}|f(x)| < |f(e_G)|.
\end{equation*}
\item[(vi)] There exists a positive continuous positive definite class function which separates the identity of $G$ 
and closed subsets $A$ with $A\not\ni e_G$.
\end{list}
\end{theorem}

\begin{proof}
(iv)$\Leftrightarrow$(vi)$\Rightarrow$(v)$\Rightarrow$(iii) and (ii)$\Rightarrow$(i) are trivial.

(iii)$\Rightarrow$(ii).
Since $G$ is first countable, there exists a countable subfamily $\{f_n\}_{n}$ of $\mathcal{F}$ which generates a neighborhood basis of the identity of $G$.
Let $(\pi_{n},\xi_n,\mathcal{H}_n)$ be the GNS triple associated to $f_n$ and $M_n$ be a von Neumann algebra generated by $\pi_n(G)$.
Since each $M_n$ is finite, the direct sum $M:=\bigoplus_{n}M_n$ is also finite and acts on a separable Hilbert space $\mathcal{H}:=\bigoplus_{n}\mathcal{H}_n$ (see the remark above Lemma \ref{yasu GNS for invariant positive definite functions}).
Put $\pi:=\bigoplus_{n}\pi_n$, 
then $\pi$ is an embedding of $G$ into $\mathcal{U}(M)$.
The image of $\pi$ is closed in $\mathcal{U}(M)$, as both $G$ and $\mathcal{U}(M)$ are Polish.

(i)$\Rightarrow$(iii).
Let $\pi$ be an embedding of $G$ into the unitary group of a finite von Neumann algebra $M$.
Since each finite von Neumann algebra is the direct sum of countably decomposable finite von Neumann algebras,
we can take of a family of countably decomposable finite von Neumann algebras $\{M_i\}_{i\in I}$ with $M=\bigoplus_{i\in I}M_i$.
In this case $\pi$ is also of the form $\pi=\bigoplus_{i\in I}\pi_i$, 
where each $\pi_i:G\rightarrow \mathcal{U}(M_i)$ is a continuous group homomorphism. 
Let $\tau_i$ be a faithful normal tracial state on $M_i$ and $(\rho_i,\xi_i,\mathcal{H}_i)$ be its GNS triple as a $C^{*}$-algebra.
Here each $\rho_i$ is an isomorphism from $M_i$ into $\mathbb{B}(\mathcal{H}_i)$ and 
\begin{equation*}
\tau_i(x) = \nai{\xi_i}{\rho_i(x)\xi_i}, \ \ \ \ x\in M_i,
\end{equation*}
holds.
Now set $f_i:=\tau_i\circ \pi_i$.
Then each $f_i$ is a continuous positive definite class functions on $G$ and 
$\{f_i\}_{i\in I}$ generates a neighborhood basis of the identity $e_G$ of $G$.

(iii)$\Rightarrow$(iv).
Let $\{f_n\}_{n}$ be a countable family of continuous positive definite class functions generating a neighborhood basis of the identity of $G$ with $f_n(e_G)=1$.
Set 
\begin{align*}
f_n'(g) &:= e^{\text{Re}(f_n(g))-1} \\
&=e^{-1}\sum_{k=0}^{\infty}\frac{1}{k!}\left[\text{Re}(f_n(g))\right]^k, 
\ \ \ \ g\in G,
\end{align*}
then $\{f_n'\}_n$ is not only a family of continuous positive definite class functions generating a neighborhood basis of the identity of $G$ with $f_n'(e_G)=1$ but also a family of positive functions.
Define a positive, continuous positive definite class function by  $f(g):=\sum_{n}f_n'(g)/2^n\ (g\in G)$.
It is easy to see that $f$ generates a neighborhood basis of the identity of $G$.
\end{proof}

\begin{remark} The proof of the above theorem is inspired by Theorem 2.1 of S. Gao \cite{Gao}.
\end{remark}
\begin{remark} Popa (Lemma 2.6 of \cite{Popa}) showed that a Polish group $G$ is of finite type if and only if it is isomorphic onto a closed subgroup of the unitary group of a separable II$_1$ factor. Therefore Theorem \ref{yasu characterization of finite type} gives a necessary and sufficient condition for a Polish group to be isomorphic onto a closed subgroup of the unitary group of a separable II$_1$ factor.  
\end{remark}
\subsection{SIN-groups and Bi-invariant Metrics}
To discuss further properties of finite type groups, we consider the following notions, say SIN-groups, bi-invariant metrics and unitarily representability. 

A neighborhood $V$ at the identity of a topological group $G$ is called {\it invariant} if it is invariant under all inner automorphisms, 
that is, $gVg^{-1}=V$ holds for all $g\in G$.
{\it A SIN-group} is a topological group which has a neighborhood basis of  the identity consisting of invariant identity neighborhoods.
Note that a locally compact Hausdorff SIN-group is unimodular.

{\it A bi-invariant metric} on a group $G$ is a metric $d$ which satisfies
\begin{equation*}
d(kg,kh) = d(gk,hk) = d(g,h), \ \ \ \ \forall g,h,k\in G.
\end{equation*}
It is known that a first countable Hausdorff topological group is SIN if and only if it admits a compatible bi-invariant metric. 

As Popa \cite{Popa} pointed out, one of the most important fact of Polish groups of finite type is 
an existence of a compatible bi-invariant metric.

\begin{lemma}\label{yasu finite type has a bi-invariant metric}
Each Polish group of finite type has a compatible bi-invariant metric.
In particular, it is SIN.
\end{lemma} 

\begin{proof}
It is enough to show that for every finite von Neumann algebra $M$ acting on a separable Hilbert space $\mathcal{H}$ the unitary group $\mathcal{U}(M)$ has a compatible bi-invariant metric.
For this let $\tau$ be a faithful normal tracial state on $M$.
Then a metric $d$ defined by
\begin{equation*}
d(u,v) :=\tau((u-v)^*(u-v))^{\frac{1}{2}}, \ \ \ \ u,v\in \mathcal{U}(M),
\end{equation*}
is a compatible bi-invariant metric on $\mathcal{U}(M)$.
\end{proof}

\subsection{Unitary Representability}
A Hausdorff topological group is called {\it unitarily representable}
if it is isomorphic as a topological group onto a subgroup 
of the unitary group of a Hilbert space.
All locally compact Hausdorff groups are unitarily representable via the left regular representation. 
It is clear that a Polish group of finite type is necessarily unitarily representable. 
The following characterization of unitary representability has been considered by specialists and can be seen in e.g., Gao \cite{Gao}.
\begin{lemma}\label{yasu characterization of unitarily representation}
For a Polish group $G$ the following are equivalent.
\begin{list}{}{}
\item[(i)] $G$ is unitarily representable.
\item[(ii)] There exists a positive, continuous positive definite function which separates the identity of $G$ 
and closed subsets $A$ with $A\not\ni e_G$.
\end{list}
\end{lemma}

\subsection{Simple Examples}

All of the following examples are well-known.
The first three examples are locally compact groups. 

\begin{example}\label{yasu compact group}
Any compact metrizable group is a Polish group of fintie type.
This follows from the Peter-Weyl theorem.
\end{example}

\begin{example}\label{yasu abelian group}
Any abelian second countable locally compact Hausdorff group is a Polish group of finite type.
Indeed its left regular representation is an embedding into the unitary group of a Hilbert space and the von Neumann algebra generated by its image is commutative (in particular, finite).
\end{example}

\begin{example}\label{yasu discrete group}
Any countable discrete group is a Polish group of finite type.
For its left regular representation is an embedding into the unitary group of a finite von Neumann algebra.  
\end{example}

The following two examples suggest there are few other examples of locally compact groups of finite type. 

\begin{example}\label{yasu ax+b group}
Let 
$G:=\left\{\left(
\begin{array}{ccc}
x & y  \\
0 & 1  
\end{array}
\right)\in GL(2,\mathbb{K})\ ;\ x\in\mathbb{K}^{\times },y\in\mathbb{K}\right\}$ 
be the $ax+b$ group, where $\mathbb{K}=\mathbb{R}$ or $\mathbb{C}$.
By easy computations, we have
\begin{equation*}
\left(
\begin{array}{ccc}
a & b  \\
0 & 1  
\end{array}
\right)
\left(
\begin{array}{ccc}
x & y  \\
0 & 1  
\end{array}
\right)
\left(
\begin{array}{ccc}
a & b  \\
0 & 1  
\end{array}
\right)^{-1}
=
\left(
\begin{array}{ccc}
x & -bx+ay+b  \\
0 & 1  
\end{array}
\right),
\end{equation*}
so that the conjugacy class 
$C\left(\left(
\begin{array}{ccc}
x & y  \\
0 & 1  
\end{array}
\right)\right)$
of 
$\left(
\begin{array}{ccc}
x & y  \\
0 & 1  
\end{array}
\right)$
is
\begin{equation*}
C\left(\left(
\begin{array}{ccc}
x & y  \\
0 & 1  
\end{array}
\right)\right) = 
\begin{cases}
\ \ \ \ 
\left\{\left(
\begin{array}{ccc}
x & \sharp   \\
0 & 1  
\end{array}
\right)\ ;\ \sharp\in\mathbb{K}\right\} & (x\not= 1), \\
\\
\ \ \ \ 
\left\{\left(
\begin{array}{ccc}
1 & \sharp   \\
0 & 1  
\end{array}
\right)\ ;\ \sharp\in\mathbb{K}^{\times }\right\} & (x=1,\ y\not= 0), \\
\\
\ \ \ \ 
\left\{\left(
\begin{array}{ccc}
1 & 0   \\
0 & 1  
\end{array}
\right)\right\} & (x=1,\ y= 0).
\end{cases}
\end{equation*}
Thus for each $n\in\mathbb{N}$ there exists a matrix $h_n\in G$ such that
$h_ng_nh_n^{-1}=\left(
\begin{array}{ccc}
1 & 1   \\
0 & 1  
\end{array}
\right)$, where $g_n:=\left(
\begin{array}{ccc}
1 & 1/n   \\
0 & 1  
\end{array}
\right)$.
Clearly, 
$g_n\rightarrow 1$
and 
$h_ng_nh_n^{-1}\not\rightarrow 
1$.
This implies that the $ax+b$ group does not admit a compatible bi-invariant metric.
Hence it is not of finite type. 

\end{example}

\begin{example}\label{yasu SL and GL}
The special linear group $SL(n,\mathbb{K})\ (n\geq 2)$ is not of finite type since the map $\left(
\begin{array}{ccc}
a & b  \\
0 & 1  
\end{array}
\right)\mapsto \left(
\begin{array}{ccc}
a & b  \\
0 & a^{-1}  
\end{array}
\right)$
is an embedding of the $ax+b$ group into $SL(2,\mathbb{K})$.
Thus the general linear group $GL(n,\mathbb{K})\ (n\geq 2)$ is also not of finite type.
\end{example}

Next we consider abelian groups.
Note that an abelian topological group is of finite type if and only if it is unitarily representable.

\begin{example}\label{yasu hilbert space}
Any separable Hilbert space is a Polish group of finite type. 
This follows from Example \ref{yasu hilbert space positive definite function} and Theorem \ref{yasu characterization of finite type}.
\end{example}

\begin{example}\label{yasu lp spaces}
A separable Banach space $l^p\ (1\leq p\leq \infty)$ is a Polish group of finite type if and only if $1\leq p\leq 2$.
The ``only if'' part follows from Example \ref{yasu l^p positive definite}
and Theorem \ref{yasu characterization of finite type}, 
but the ``if'' part is non-trivial. 
For details, see \cite{Meg}.
\end{example}

Here is another counter example.

\begin{example}\label{yasu c01 space}
Separable Banach space $C[0,1]$ of all continuous functions on the interval $[0,1]$ is a Polish group but not of finite type.
For, since every separable Banach space is isometrically isomorphic to a closed subspace of $C[0,1]$, 
if $C[0,1]$ is of finite type, 
then any separable Banach space is a Polish group of finite type.
But this is a contradiction to the previous example.
\end{example}

\subsection{Application to Locally Compact Groups}
It is known that a second countable locally compact group is of finite type if and only if it is a SIN-group (see e.g., Theorem 13.10.5 of J. Dixmier \cite{Dixmier}). We give a new proof of this fact using Theorem \ref{yasu characterization of finite type}. We thank the referee for letting us know the above literature.

\begin{proposition}\label{yasu locally compact}
A second countable locally compact Hausdorff group is of finite type
if and only if it is SIN. 
\end{proposition}

\begin{proof}
Let $G$ be a second countable locally compact Hausdorff SIN-group, $\mu$ be the Haar measure on it
and $\lambda$ be its left-regular representation.
For each compact invariant neighborhood $U$ of the identity, 
we define a continuous positive definite function $\varphi_U$ on $G$ by
\begin{equation*}
\varphi_U(g) := \nai{\chi_U}{\lambda(g)\chi_U } 
= \mu\left(U\cap gU\right), \ \ \ \ g\in G.
\end{equation*}
Note that, for each $g,h,x\in G$, we have
\begin{equation*}
h^{-1}x \in U \Leftrightarrow x\in hU=Uh \Leftrightarrow xh^{-1}\in U,
\end{equation*}
and
\begin{equation*}
(gh)^{-1}x\in U \Leftrightarrow x\in ghU=gUh \Leftrightarrow xh^{-1}\in gU.
\end{equation*}
Also note that 
a locally compact SIN-group is unimodular.
Thus we see that
\begin{align*}
\varphi_U(h^{-1}gh) &= \nai{\lambda(h)\chi_U}{\lambda(gh)\chi_U } \\
&= \int_{G}\chi_U(h^{-1}x)\chi_U((gh)^{-1}x)d\mu(x) \\
&= \int_{G}\chi_U(xh^{-1})\chi_{gU}(xh^{-1})d\mu(x) \\
&= \int_{G}\chi_U(x)\chi_{gU}(x)d\mu(x) \\
&= \int_{G}\chi_U(x)\chi_{U}(g^{-1}x)d\mu(x) \\
&=\varphi_U(g).
\end{align*}
This implies $\varphi_U$ is a class function.
It is not hard to check that a family $\{\varphi_U\}_U$ 
generates a neighborhood basis of the identity of $G$.
This completes the proof by Theorem \ref{yasu characterization of finite type}.
\end{proof}


\begin{remark}\label{yasu locally compact remark}
(1) R. V. Kadison and I. Singer \cite{KadSin} proved that every connected locally compact Hausdorff SIN group is isomorphic as a topological group onto a topological group of the form $\mathbb{R}^n\times K$, where $K$ is a compact Hausdorff group.\\
(2) K. Hofmann, S. Morris and M. Stroppel \cite{hofmann} proved that every totally disconnected locally compact Hausdorff group is SIN if and only if it is a strict projective limit of discrete groups.
\end{remark}

\subsection{A Characterization for Amenable Groups}
Next, we characterize (not necessarily locally compact) amenable Polish groups of finite type. Recall that a Hausdorff topological group $G$ is amenable if ${\rm LUCB}(G)$ admits a left-translation invariant
positive functional $m\in {\rm LUCB}(G)^*$ with $m(1)=1$,
where ${\rm LUCB}(G)$ is a complex Banach space of all left-uniformly continuous bounded functions on $G$.
Such a $m$ is called an {\it invariant mean}. 

\begin{theorem}\label{yasu amenable group}
A unitarily representable amenable Polish group is of finite type
if and only if it is SIN.
\end{theorem}

\begin{proof}
Let $G$ be a unitarily representable amenable Polish SIN-group
and let $f$ be a positive, continuous positive definite function on $G$ which separates the identity of $G$ 
and closed subsets $A$ with $A\not\ni e_G$ 
(see Lemma \ref{yasu characterization of unitarily representation} ).
We may and do assume $f(e_G)=1$.
For each $x\in G$, we define a positive function $\Psi_{x,f}:G\to [0,1]$ by
\begin{equation*}
\Psi_{x,f}(g) := f(g^{-1}xg), \ \ \ \ g\in G.
\end{equation*}
We show that $\Psi_{x,f}\in {\rm LUCB}(G)$. Fix an arbitrary $\varepsilon>0$. Since the positive definite function $f$ is left-uniformly continuous, there exists a neighborhood $V$ of $e_G$ such that 
\[|f(g)-f(h)|<\varepsilon\]
holds whenver $g,h\in G$ satisfy $g^{-1}h\in V$. There exists a neighborhood $W$ of $e_G$ such that $W=W^{-1}$ and $W\cdot W\subset V$ holds. Since $G$ is SIN, there exists an invariant neighborhood $U$ of $e_G$ with $U\subset W$. Let $g,h\in G$ satisfy $g^{-1}h\in U$. By the invariance of $U$, it holds that $h\in gU=Ug$ and therefore that $hg^{-1}\in U$. Then we see that
\eqa{
(h^{-1}xh)^{-1}(g^{-1}xg)^{-1}&=h^{-1}x^{-1}hg^{-1}xg\in h^{-1}x^{-1}Uxg\\
&= h^{-1}Ug=Uh^{-1}g\\
&=U(g^{-1}h)^{-1}\subset W\cdot W^{-1}\\
&\subset V,
}
which implies 
\[|\Psi_{x,f}(h)-\Psi_{x,f}(g)|=|f(h^{-1}xh)-f(g^{-1}xg)|<\varepsilon.\]
Hence $\Psi_{x,f}$ is left-uniformly continuous and we have $\Psi_{x,f}\in {\rm LUCB}(G)_+$.  
Let $m\in {\rm LUCB}(G)^*$ be an invariant mean.
Put 
\begin{equation*}
\psi_f(x) := m(\Psi_{x,f}), \ \ \ \ x\in G,
\end{equation*}
then $\psi_f(x)$ is clearly a positive, positive definite class function on $G$ with $\psi_f(e_G)=1$. We show that $\psi_f$ is continuous. Since $m$ is continuous, it suffices to show that $G\ni x\mapsto \Psi_{x,f}\in {\rm LUCB}(G)_+$ is continuous. Let $x,y\in G$. By Krein's inequality, we have
\eqa{
||\Psi_{x,f}-\Psi_{y,f}||^2&=\sup_{g\in G}|f(g^{-1}xg)-f(g^{-1}yg)|^2\\
&\le 2\sup_{g\in G}|1-\text{Re}f(g^{-1}x^{-1}yg)|\\
&=2\sup_{g\in G}|1-f(g^{-1}x^{-1}yg)|.
}
Fix $\varepsilon>0$. Since $f$ is left-uniformly continuous, there exists an invariant neighborhood $V$ of $e_G$ such that 
$|f(x)-f(y)|<\varepsilon$ holds for $x,y\in G$ with $x^{-1}y\in G$. Then for $x,y\in G$ with $x^{-1}y\in V$, we have
$g^{-1}x^{-1}yg\in g^{-1}Vg=V$. Therefore it holds that
\[|1-f(g^{-1}x^{-1}yg)|=|f(e_G)-f(g^{-1}x^{-1}yg)|<\varepsilon.\]
Hence we have 
\[||\Psi_{x,f}-\Psi_{y,f}||^2\le 2\varepsilon.\]
Therefore $G\ni x\mapsto \Psi_{x,f}\in \text{LUCB}(G)_+$ is continuous, hence so is $\psi_f$. 
We next show that $\psi_f$ separates the identity of $G$ 
and closed subsets $A$ with $A\not\ni e_G$. Fix such a closed set $A$. Since $A^c=G\setminus A$ is an open neighborhood of $e_G$, there exists an open invariant neighborhood $V$ of $e_G$ contained in $A^c$. Then we have $A\subset V^c$ and $e_G\notin V$. Since $f$ separates $e_G$ and $V^c$, we have 
\[\delta:=\sup_{g\in V^c}|f(g)|<1.\]
It then follows, by the invariance of $V^c$, that for $x\in V^c$, 
\[
||\Psi_{x,f}||=\sup_{g\in G}|f(g^{-1}xg)|\le \sup_{g\in V^c}|f(g)|\le \delta,
\]
which implies
\[\sup_{x\in A}|\psi_f(x)|\le \sup_{x\in V^c}|\psi_f(x)|=\sup_{x\in V^c}|m(\Psi_{x,f})|\le \sup_{x\in V^c}||\Psi_{x,f}||\le \delta<1.\]
Therefore $\psi_{f}$ separates $A$ and $e_G$.    
This completes the proof by Theorem \ref{yasu characterization of finite type}.
\end{proof}

\begin{remark}The above proof is inspired by the proof of Theorem 2.13 of J. Galindo \cite{Galindo}.
\end{remark}

\section{More Examples of Finite Type Groups}
In this section we will give nother examples of Polish groups of finite type.  
To construct such examples we need to start not from finite von Neumann algebras, but from semifinite von Neumann algebras, say of type I$_{\infty}$ or of type II$_{\infty}$. In the end of this section we also review other known examples of Polish groups of finite type.
\subsection{$L^2$-unitary groups $\mathcal{U}(M)_2$}
Let $M$ be a semifinite von Neumann algebra on a Hilbert space $\mathcal{H}$ equipped with a normal faithful semifinite trace $\tau$. A densely defined, closed operator $T$ on $\mathcal{H}$ is said to be {\it affiliated} to $M$ if for all $u\in \mathcal{U}(M')$, $uTu^*=T$ holds. Denote by $\overline{M}$ the set of all densely defined, closed operators on $\mathcal{H}$ which are affiliated to $M$. 
Recall that $L^2(M,\tau)$ is a Hilbert space completion of the space $\mathfrak{n}_{\tau}:=\{x\in M; \tau(x^*x)<\infty\}$ by the inner product 
\[\nai{x}{y}:=\tau(x^*y),\ \ \ \ x,y\in \mathfrak{n}_{\tau}.\]
We define $||x||_2:=\tau(x^*x)^{\frac{1}{2}}$ for $x\in L^2(M,\tau)$. 
\begin{definition}
We call $\mathcal{U}(M)_2:=\{u\in \mathcal{U}(M); 1-u\in L^2(M,\tau)\}$ the {\it $L^2$-unitary group} of $(M,\tau)$.
\end{definition}
Note that when $M$ is not a factor, $\mathcal{U}(M)_2$ depends on the choice of $\tau$ too. In the sequel we show the following theorem.
\begin{theorem}\label{3U(M)_2} Let $M$ be a separable semifinite von Neumann algebra with a normal faithful semifinite trace $\tau$. Then 
$\mathcal{U}(M)_2$ is a Polish group of finite type,  where the topology is determined by the following metric $d$,
\[d(u,v):=||u-v||_2,\ \ \ \ u,v\in \mathcal{U}(M)_2.\]
\end{theorem}
To prove the theorem, we need some preparations. 
In the sequel we consider $M$ to be represented on $\mathcal{H}=L^2(M,\tau)$ by left multiplication. Recall that a closed operator $T\in \overline{M}$ on $L^2(M,\tau)$ is called $\tau$-{\it measurable} if for any $\varepsilon>0$, there exists a projection $p\in M$ with $\text{ran}(p)\subset \dom{T}$ and $\tau(1-p)<\varepsilon$. 
Note that $L^2(M,\tau)$ can be identified with the set of closed, densely defined and $\tau$-measurable operators $T$ such that 
\[||T||_2^2:=\tau(|T|^2)=\int_0^{\infty}\lambda^2d\tau(e(\lambda))<\infty,\]
 where $e(\cdot )$ is a spectral resolution of $|T|=(T^*T)^{\frac{1}{2}}$ and $T=u|T|$ is the polar decomposition of $T$ (for more details about non-commutative integration, see vol II of \cite{Tak}). 

\begin{lemma}\label{hiro L2 topological group}
Let $M$ be a semifinite von Neumann algebra with a normal faithful semifinite trace $\tau$. Then $\mathcal{U}(M)_2$ is a topological group.
\end{lemma}
\begin{proof}
This can be shown directly, using the equalities: 
\[||x^*||_2=||x||_2,\ \ \ \ ||uxv||_2=||x||_2,\]
for all $x\in L^2(M,\tau)$ and $u,v\in \mathcal{U}(M)$.
\end{proof}
\begin{lemma}\label{3dense} Let $M$ be a semifinite von Neumann algebra with a normal faithful semifinite trace $\tau$. Let $U$ be a densely defined closed $\tau$-measurable operator on $L^2(M,\tau)$ affiliated to $M$. Then $\dom{U}\cap M$ is dense in $L^2(M,\tau)$.
\end{lemma}
\begin{proof} Let $\varepsilon>0$. Let $\xi \in L^2(M,\tau)$. Since $M\cap L^2(M,\tau)$ is dense, there exists $\xi_0 \in M\cap L^2(M,\tau)$ such that $||\xi-\xi_0||_2<\varepsilon$. On the other hand, the measurability of $U$ implies the existence of an increasing sequence $\{p_n\}_{n=1}^{\infty}$ of projections in $M$ such that $p_nL^2(M,\tau)\subset \dom{U}$ for all $n$ and $p_n\nearrow 1$ strongly. Therefore there exists $n_0\in \mathbb{N}$ such that 
\[||\xi_0-p_{n_0}\xi_0||_2<\varepsilon.\]
By the choice of $\xi_0$, $p_{n_0}\xi_0\in \dom{U}\cap M$ and 
\eqa{
||\xi-p_{n_0}\xi_0||_2&\le ||\xi-\xi_0||_2+||\xi_0-p_{n_0}\xi_0||_2\\
&\le \varepsilon +\varepsilon=2\varepsilon.
}
Since $\varepsilon$ is arbitrary, it follows that $\dom{U}\cap M$ is dense in $L^2(M,\tau)$.
\end{proof}
\begin{lemma}\label{hiro d is comp}
Let $M$ be a semifinite von Neumann algebra with a normal faithful semifinite trace $\tau$. $d$ is a complete metric on $\mathcal{U}(M)_2$.
\end{lemma}
\begin{proof} Suppose $\{u_n\}_{n=1}^{\infty}$ is a $d$-Cauchy sequence in $\mathcal{U}(M)_2$. Since $L^2(M,\tau)$ is complete, there exists $V\in L^2(M,\tau)$ such that $||(1-u_n)-V||_2\to 0$. Define $U:=1-V$. Then $||U-u_n||_2\to 0$. We show that $U$ is bounded and moreover $U\in \mathcal{U}(M)_2$. Since $U$ is closed and $\dom{U}\cap M$ is dense by Lemma \ref{3dense}, to prove the boundedness of $U$ it suffices to show that $U$ is isometric on $\dom{U}\cap M$. Let $\xi \in \dom{U}\cap M$. Since $\xi$ is bounded, we have
\eqa{
||(U-u_n)\xi||_2^2&=\tau(\xi^*(U-u_n)^*(U-u_n)\xi)\\
&=\tau((U-u_n)\xi \xi^*(U-u_n)^*)\\
&\le ||\xi||^2\tau ((U-u_n)(U-u_n)^*)\\
&=||\xi||^2||U-u_n||_2^2\to 0,
}
which implies 
\[||U\xi||_2=\lim_{n\to \infty}||u_n\xi||_2=||\xi||_2,\]
for all $\xi \in \dom{U}\cap M$. 
Therefore $U|_{\dom{U}\cap M}$ is isometric and $U$ is bounded. Since $||U^*-u_n^*||_2=||U-u_n||_2$, it holds that $U^*$ is an isometry too, which means $U$ is a unitary. Finally, it is clear that $U=1-V\in \mathcal{U}(M)_2$. 
\end{proof}
\begin{proof}[Proof of Theorem \ref{3U(M)_2}]
Since $M$ is separable, the separability of $\mathcal{U}(M)_2$ follows from the separability of $L^2(M,\tau)$. Therefore by Lemma \ref{hiro d is comp}, $\mathcal{U}(M)_2$ is a Polish group. By Schoenberg's theorem (see Example \ref{yasu hilbert space positive definite function}), 
\[\varphi(u):=e^{-||1-u||_2^2},\ \ \ \ u\in \mathcal{U}(M)_2,\]
is a continuous, positive definite class function on $\mathcal{U}(M)_2$. 
It is easy to see that $\varphi$ generates a neighborhood basis of the identity of $\mathcal{U}(M)_2$. Therefore the claim follows from Theorem \ref{yasu characterization of finite type}. 
\end{proof}
\begin{remark}
$\mathcal{U}(M)_2''=M.$
\end{remark}
\begin{proof}
Clearly $\mathcal{U}(M)_2''\subset M$. Let $p$ be a finite projection in $M$. Then $2p\in L^2(M,\tau)$ and $1-2p\in \mathcal{U}(M)_2$. Therefore $p\in \mathcal{U}(M)_2''$. Since $M$ is semifinite, $M$ is generated by finite projections. Therefore $\mathcal{U}(M)_2''=M$.
\end{proof}
When $M=\mathbb{B}(\mathcal{H})$,\ $\mathcal{U}(M)_2$ is the well-known example of a Hilbert-Lie group and is denoted as $\mathcal{U}(\mathcal{H})_2$. 
\subsection{Non-isomorphic Properties of $\mathcal{U}(M)_2$}
J. Feldman \cite{Feldman} gave a complete description of a group isomorphism between the unitary groups of type II$_1$ von Neumann algebras. In particular, in the proof of Theorem 4 of \cite{Feldman}, he uses the following simple observation: let $p$ be a projection in a von Neumann algebra $M$, then $u_p:=1-2p$ is a self-adjoint unitary in $M$. Using this correspondence, he deduced that the group isomorphism $\pi: \mathcal{U}(M_1)\to \mathcal{U}(M_2)$ between type II$_1$ von Neumann algebras $M_1,M_2$ induces an order isomorphism between their projection lattices, thereby proving that the isomorphism $\pi$ is lifted to a ring *-isomorphism $\overline{\pi}:M_1\to M_2$ (which may not preserve the scalar multiplication) in such a way that 
\[\overline{\pi}(u)=\theta(u)\pi(u),\ \ \text{for all }u\in \mathcal{U}(M_1)\]
holds, where $\theta$ is a multiplicative map from $\mathcal{U}(M_1)$ to $Z(\mathcal{U}(M_2))$.    
Let $\mathcal{H}$ be an infinite dimensional Hilbert space. Using his idea, we show that when $M$ is a II$_{\infty}$ factor and $N$ is a finite von Neumann algebra, then $\mathcal{U}(M)_2$, $\mathcal{U}(\mathcal{H})_2$ and $\mathcal{U}(N)$ are mutually  non-isomorphic. In this subsection, no separability assumptions are required.
\begin{proposition}\label{U(M)_2 is not U(H)_2}
Let $M$ be a II$_{\infty}$ factor. Then $\mathcal{U}(M)_2$ is not isomorphic onto $\mathcal{U}(\mathcal{H})_2$.
\end{proposition}
\begin{proof}Let $\tau$ be a normal faithful semifinite trace on $M$, Tr be the usual operator trace on $\mathcal{H}$. We denote their corresponding trace 2-norms by $||\cdot ||_{2,\tau}$ and $||\cdot ||_{2,\text{Tr}}$, respectively.
We prove the claim by contradiction. Suppose there exists a topological group isomorphism $\varphi : \mathcal{U}(M)_2\to \mathcal{U}(\mathcal{H})_2$. Let $p$ be a nonzero finite-rank projection in $\mathbb{B}(\mathcal{H})$. Then $1-2p\in \mathcal{U}(\mathcal{H})_2$ and let
\[q:=\frac{1}{2}(1-\varphi^{-1}(1-2p)).\]
It is easy to see that $q\in L^2(M,\tau)$ is a nonzero finite projection in $M$. 
Let $k\in \mathbb{N}$. Since $M$ is a II$_{\infty}$ factor, there exists a projection $0<q_k\le q$ in $M$ such that $\lim_{k\to \infty}\tau(q_k)=0$. Define $p_k:=\frac{1-\varphi(1-2q_k)}{2}$. Since 
\[||q_k||_{2,\tau}^2=\tau(q_k)\to 0\ (k\to \infty),\]
$1-2q_k\to 1$ holds in $\mathcal{U}(M)_2$, which in turn means 
\[1-2p_k=\varphi(1-2q_k)\to \varphi(1)=1 \ \text{ in }\mathcal{U}(\mathcal{H})_2.\]
However, since the topology of $\mathcal{U}(\mathcal{H})_2$ is given by the operator trace 2-norm, it holds that
\[2\le ||2p_k||_{2,\text{Tr}}=||1-(1-2p_k)||_{2,\text{Tr}}\to 0\ (k\to \infty).\]
This is clearly a contradiction. Therefore $\mathcal{U}(M)_2\not \cong \mathcal{U}(\mathcal{H})_2$.
\end{proof}
\begin{proposition}\label{hiro U(M)_2 not iso to U(N)}
Let $M$ be a type I$_{\infty}$ or type II$_{\infty}$ factor, $N$ be a finite von Neumann algebra. Then $\mathcal{U}(M)_2$ is not isomorphic onto $\mathcal{U}(N)$. 
\end{proposition}
\begin{proof}
Let $\tau$ be a normal faithful semifinite trace on $M$. Let $u\in Z(\mathcal{U}(M)_2)$ be an element of $\mathcal{U}(M)_2$ which commutes with every element in $\mathcal{U}(M)_2$. Then for any finite projection $p\in M$, $u(1-2p)=(1-2p)u$ holds. Therefore $u$ commutes with all finite projections in $M$. Since $M$ is generated by its finite projections, $u\in Z(M)=\mathbb{C}1$ holds. Since $u-1\in L^2(M,\tau)$, this forces $u=1$. Therefore the center of $\mathcal{U}(M)_2$ is $\{1\}$, while the center of $\mathcal{U}(N)$ contains $\mathbb{C}1$. 
\end{proof}
\begin{remark}
We thank the referee for telling us the above simple proof and the literature \cite{Feldman}.
\end{remark}

%

\subsection{Other Known Examples}
The class $\mathscr{U}_{\text{fin}}$ has not been studied well. 
However, there are some known examples other than the ones presented in $\S 2.6$.
\begin{example}Normalizer groups $\mathcal{N}_M(A)$ and $\mathcal{N}(E)$\\
Let $A$ be an abelian von Neumann subalgebra of a separable II$_1$ factor $M$. The {\it normalizer group} $\mathcal{N}_M(A)$ of $A$,  defined by 
\[\mathcal{N}_M(A):=\{u\in \mathcal{U}(M); uAu^*=A\},\]
is clearly a strongly closed subgroup of $\mathcal{U}(M)$ and hence belongs to $\mathscr{U}_{\text{fin}}$. This group has been drawn much attention to specialists, especially when $A$ is maximal abelian and $\mathcal{N}_M(A)$ generates $M$ as a von Neumann algebra. In such a case, $A$ is called a {\it Cartan subalgebra}. Similarly, the {\it normalizer group} $\mathcal{N}(E)$ for a normal faithful conditional expectation $E:M\to N$ onto a von Neumann subalgebra $N$,
\[\mathcal{N}(E):=\{u\in \mathcal{U}(M);uE(x)u^*=E(uxu^*), \text{ for all }x\in M\}\]
is also of finite type.
\end{example}
\begin{example}The full group $[\mathcal{R}]$\\
Let $\mathcal{R}$ be a II$_1$ countable equivalence relation on a standard probability space $(X,\mu)$. A. Furman showed that the full group $[\mathcal{R}]$ equipped with so-called {\it uniform topology} is a Polish group of finite type (see $\S 2$ of Furman \cite{Furman}).
\end{example}

\section{Hereditary Properties of Finite Type Groups}
In this section, we discuss several permanence properties of the class $\mathscr{U}_{\text{fin}}$ under several algebraic operations. 
In summary, we will observe the following permanence properties of finite type groups. 
\begin{center}
\begin{tabular}{|c||c|} \hline
Operation & $\mathscr{U}_{\text{fin}}$? \\\hline
Closed subgroup $H<G$ & YES \\\hline
Countable direct product $\prod_{n\ge 1}G_n$ & YES \\\hline
Semidirect product $G\rtimes H$ & NO \\\hline
Quotient $G/N$ & NO \\\hline
Extension $1\to N\to G\to K\to 1$ & NO \\\hline
Projective limit $\displaystyle \lim_{\longleftarrow }G_n$ & YES \\\hline
\end{tabular}
\end{center}
As can be seen from the above table, finiteness property is delicate and can easily be broken under natural operations. 
\begin{remark} (On the ultraproduct of metric groups) Let $\{(G_n,d_n)\}_{n=1}^{\infty}$ be a sequence of finite type Polish groups with a compatible bi-invariant metric. It is not difficult to show that the ultraproduct $(G_{\omega},d_{\omega})$ of $\{(G_n,d_n)\}_{n=1}^{\infty}$ along a free ultrafilter $\omega \in \beta \mathbb{N}\setminus \mathbb{N}$ is a completely metrizable topological group of finite type, but not Polish in general. We will discuss topological groups which are embeddable into the unitary group of a (not necessarily separable) finite von Neumann algebra elsewhere. 
\end{remark}

\subsection{Closed Subgroup and Countable Direct Product}
It is clear the class $\mathscr{U}_{\text{fin}}$ is closed under taking a closed (or even $G_{\delta}$) subgroup. Since a countable direct sum of separable finite von Neumann algebras is again separable and finite, the class $\mathscr{U}_{\text{fin}}$ is closed under  countable direct product. 
\subsection{Extension and Semidirect Product}
The class $\mathscr{U}_{\text{fin}}$ is not closed under extension nor semidirect product.
\begin{proposition}
There exits a Polish group $G$ not of finite type, which has a closed normal subgroup $N$ such that $N$ and the quotient group $G/N$ are of finite type.
\end{proposition}
\begin{proof}
Let $G$ be the $ax+b$ group (see Example \ref{yasu ax+b group}). 
Since $G$ does not have a compatible bi-invariant metric, it is not of finite type. 
On the other hand, $G$ can be written as a semidirect product $G=\mathbb{K}\rtimes \mathbb{\mathbb{K}}^{\times}$, where $\mathbb{K}^{\times}$ acts on $\mathbb{K}$ as a multiplication. There fore the exact sequence 
\[0\longrightarrow \mathbb{\mathbb{K}}\longrightarrow G\longrightarrow \mathbb{K}^{\times}\longrightarrow 1\]
gives a counter example for extension case.
\end{proof}
Note that the above example also shows that the class $\mathscr{U}_{\text{fin}}$ is not closed under semidirect product. 
\subsection{Quotient}
The class $\mathscr{U}_{\text{fin}}$ is not closed under quotient.
\begin{proposition}\label{4quotient}
There exists an abelian Polish groups of finite type $G$ such that the quotient $G/N$ of $G$ by its closed subgroup is not of finite type.  
\end{proposition}
\begin{proof}
Consider the separable Banach space $A:=l^3$ as an additive Polish group. As we saw in Example \ref{yasu lp spaces}, $l^p (1\le p\le \infty)$ is unitarily representable if and only if $1\le p\le 2$. On the other hand, every separable Banach space is isomorphic onto a quotient Banach space of $\ell^1$ (see e.g., Theorem 5.1 of \cite{Fabian}). In particular, although not of finite type, $A=\ell^3$ is a quotient of $G:=\ell^1$ by its closed subgroup $N$. 
\end{proof} 
\begin{remark}
Note that even for abelian Polish groups, the situation can be worst possible. It is known (chapter 4 of \cite{Ban83}) that there exists an abelian Polish group $A$ which has no non-trivial unitary representation. Such a group is called {\it strongly exotic}.  On the other hand, S. Gao and V. Pestov \cite{GaoPestov} proved that any abelian Polish group is a quotient of $\ell^1$ by a closed subgroup $N$. Therefore, strongly exotic groups are also quotients of finite type Polish groups. 
\end{remark}
\subsection{Projective Limit}
The class $\mathscr{U}_{\text{fin}}$ is closed under projective limit. 
\begin{proposition}\label{prolim}
Let $\{G_n, j_{m,n}:G_m\to G_n (n\le m)\}_{n,m=1}^{\infty}$ be a projective system of Polish groups of finite type. 
Then $\displaystyle G=\lim_{\longleftarrow }G_n$ is a Polish group of finite type.
\end{proposition}
\begin{proof}
Since the connecting map $\{j_{m,n}\}$ is continuous, it is clear that $G$ can be seen as a closed subgroup of $\prod_{n \in \mathbb{N}}G_{n}$. Since finiteness property passes to direct product, $\prod_{n \in \mathbb{N}}G_{n}$ is also a Polish group of finite type. Therefore its closed subgroup $G$ is also a Polish of finite type. 
\end{proof}

\section{Some Questions}
Finally let us discuss some questions to which we do not have answers at this stage.  
Let $\mathscr{U}_{\text{inv}}$ denote the class of Polish groups with a compatible bi-invariant metric. 
As we saw in Example \ref{yasu l^p positive definite}, $\mathscr{U}_{\text{inv}}$ is strictly larger than $\mathscr{U}_{\text{fin}}$ ($l^3$ is in $\mathscr{U}_{\text{inv}}$ but not in $\mathscr{U}_{\text{fin}}$). Therefore the unitarily representability is indispensable (this was also pointed out by Popa). Furthermore, there exists a more interesting example. Recently L. van den Dries and S. Gao \cite{GaoDries} constructed a Polish group $G$ with a compatible bi-invariant metric, which does not have Lie sum (see \cite{GaoDries} for the definition). On the other hand, we proved in \cite{AndoMats} that if $G$ belongs to the class $\mathscr{U}_{\text{fin}}$, then $G$ has Lie sum. Thus $G$ is not of finite type. 
Therefore it would be desirable to consider the following questions (the latter was posed in Popa\cite{Popa}, $\S 6.5$):
\begin{question} Is van den Dries-Gao's Polish group unitarily representable?
\end{question}

\begin{question}[Popa]\label{hiro question}
Is a unitarily representable Polish SIN-group of finite type? 
\end{question}

Hopefully Theorem \ref{yasu characterization of finite type} will play the role for solving the above questions.
Also, since $l^p$ belongs to $\mathscr{U}_{\text{fin}}$ if and only if $1\le p\le 2$, it is worth considering whether
\begin{question}
Let $\mathcal{H}$ be a separable infinite-dimensional Hilbert space. Does $\mathcal{U}(\mathcal{H})_p:=\{u\in \mathcal{U}(\mathcal{H}); 1-u\in S^p(\mathcal{H})\}$ belong to $\mathscr{U}_{\text{fin}}$ for some $1\le p<2$ ? Here $S^p(\mathcal{H})$ denotes the space of Schatten $p$-class operators.
\end{question}

Finally, let us remind that there is another candidate for a counterexample to Question \ref{hiro question}. 
Recall that a finite von Neumann algebra $N$ equipped with a normal faithful tracial state $\tau$ is said to have {\it property (T)} if for each $\varepsilon>0$, there exists a finite set $\mathcal{F}\subset N$ and $\delta>0$ with the property that whenever $\varphi:N\to N$ is a unital completely positive $\tau$-preserving map satisfying $||\varphi(x)-x||_2<\delta$ for all $x\in \mathcal{F}$, then $||\varphi(a)-a||_2\le \varepsilon ||a||$ holds for all $a\in N$. 

Let $M$ be a separable II$_1$ factor with property (T), ${\rm Aut}(M)$ be a Polish group of all *-automorphisms of $M$ equipped with the pointwise $||\cdot ||_2$-convergence topology. Due to the property (T), this topology coincides with the topology of uniform $||\cdot ||_2$-convergence on the closed unit ball $M_1$. Since the latter topology is given by the bi-invariant metric $d$ defined by
\[d(\alpha,\beta):=\sup_{x\in M_1}||\alpha(x)-\beta(x)||_2,\ \ \ \alpha,\beta \in {\rm Aut}(M),\]
${\rm Aut}(M)$ is a Polish SIN-group. By considering the standard representation, ${\rm Aut}(M)$ is unitarily representable as well. Therefore it would be interesting to check if ${\rm Aut}(M)$ is actually of finite type or not. 
\subsection*{Acknowledgements}
The authors would like to express their thanks to Professor Asao Arai, Professor Uffe Haagerup, Professor Izumi Ojima and Professor Konrad Schm\"{u}dgen for their important comments, discussions and supports. We also thank to Mr. Takahiro Hasebe, Mr. Ryo Harada, Mr. Kazuya Okamura and Dr. Hayato Saigo for the discussions in a seminar and for their continual interests in our work. The second named author thanks to Mr. Abel Stolz for his kind discussions.  
Final version of the paper was done during authors' visit to the conference ``Von Neumann algebras and ergodic theory of group actions 2011" at Institut Henri Poincar\'e. They thank the organizers Professor Damien Gaboriau,  Professor Sorin Popa and Professor Stefaan Vaes for their fruitful discussions and also to Professor Jesse Peterson for informing us about the Popa and his recent results concerning cocycle superrigidity and the class $\mathscr{U}_{\text{fin}}$.
Last but not least, we thank the anonymous referee for telling us the literature and suggesting a simpler proof of some of the main results.  
Both of the authors are supported by Research fellowships of the Japan Society for the Promotion of Science for Young Scientists.

\end{document}